\documentclass[12pt]{article}

\usepackage{amsmath}
\allowdisplaybreaks[4]

\usepackage[all]{xy}
\usepackage{amssymb}
\usepackage{amsthm}
\usepackage{hyperref}
\hypersetup{colorlinks=true,linkcolor=blue,citecolor=red}
\usepackage{amsmath}
\usepackage{amscd}
\usepackage{verbatim}
\usepackage{eurosym}
\usepackage{float}
\usepackage{color}
\usepackage{dcolumn}
\usepackage[mathscr]{eucal}
\usepackage[all]{xy}
\usepackage{hyperref}
\usepackage{mathrsfs}
\usepackage{amsmath}
\usepackage{amssymb}
\usepackage{amsfonts,ifpdf}
\usepackage{graphicx}
\usepackage{times}
\usepackage{float}
\usepackage{epstopdf}
\usepackage{cite}
\usepackage{youngtab}
\usepackage{ytableau}
\ytableausetup
{mathmode, boxsize=0.9em}

\usepackage{graphicx}

\makeatletter
\newcommand*\bigcdot{\mathpalette\bigcdot@{.6}}
\newcommand*\bigcdot@[2]{\mathbin{\vcenter{\hbox{\scalebox{#2}{$\m@th#1\bullet$}}}}}
\makeatother

\newlength{\boxedparwidth}
  {\begin{center} \begin{tabular}{|@{\hspace{.315in}}c@{\hspace{.15in}}|}
                  \hline \\ \begin{minipage}[t]{\boxedparwidth}
                  \setlength{\parindent}{.25in}}%
  {\end{minipage} \\ \\ \hline \end{tabular} \end{center}}

\allowdisplaybreaks

\newcommand{\mbz}{\mathbb{Z}}

\newcommand{\mbn}{\mathbb{N}}

\newcommand{\mbq}{\mathbb{Q}}

\newcommand{\be}{\textbf{e}}

\DeclareMathOperator{\SL}{SL}

\DeclareMathOperator{\ord}{ord}

\def \c {NC}

\allowdisplaybreaks

\newtheorem{thm}{Theorem}[section]
\newtheorem{prop}[thm]{Proposition}

\newtheorem{lem}[thm]{Lemma}
\newtheorem{cor}[thm]{Corollary}

\numberwithin{equation}{section}

\begin{document}

\baselineskip 20pt

\begin{center}
{\Large\bf Congruences for modular forms and applications to crank functions}\\ [7pt]
\end{center}

\vskip 3mm

\begin{center}
{Hao Zhang}$^{1,2}$ and {Helen W.J. Zhang}$^{1,2}$
\\[8pt]
$^{1}$School of Mathematics\\
Hunan University\\
Changsha 410082, P. R. China\\[12pt]

$^{2}$Hunan Provincial Key Laboratory of \\
Intelligent Information Processing and Applied Mathematics\\
Changsha 410082, P. R. China\\[15pt]

Emails:  zhanghaomath@hnu.edu.cn,   helenzhang@hnu.edu.cn
\\[15pt]

\end{center}

\vskip 3mm
\begin{abstract}

In this paper, motivated by the work of Mahlburg, we find congruences for a large class of modular forms. Moreover, we generalize the generating function of the Andrews-Garvan-Dyson crank on partition and establish several new infinite families of congruences.
In this framework,
we showed that both the birank of an ordered
pair of partitions introduced by Hammond and Lewis,
and $k$-crank of $k$-colored partition introduced by Fu and Tang process the same as the partition function and crank.

\vskip 6pt

\\ [7pt]
\end{abstract}

\section{Introduction}
The objective of this paper is to present a modular form approach to deriving congruences of the crank-type generating functions.
This work was inspired by the key construction of
Mahlburg \cite{Mahlburg-2005} that
the generating functions of crank is deeply related to Klein forms and weakly holomorphic modular forms of half-integral weight on the congruence subgroup $\Gamma_1(l^j)$.
This leads to arithmetic properties of the crank of partitions, which confirmed a conjecture of Ono \cite{Mahlburg-2006}.

Recall that a partition of a nonnegative integer $n$ is any nonincreasing sequence of positive integers whose sum is $n$.
As usual, let $p(n)$ denote the number of partitions of $n$.
A breakthrough of applying modular forms to partition functions is due to Ono \cite{Ono-2000}.
Along with this direction,
the congruence properties for partition function are deeply investigated in the subsequential work .
For example, Ahlgren and Ono \cite{Ahlgren-Ono-2001} showed that
for prime $\ell \geq 5$ and $m>0$, there are infinitely many non-nested arithmetic progressions $A n+B$ such that for every integer $n$ we have
\[p(An+B) \equiv 0\pmod{\ell^{m}}.\]
See \cite{Folsom-2014,Garthwaite-Jameson-2021,Radu-2012} for more results.

To give the combinatorial interpretations of the well-known Ramanujan's congruences,
Dyson \cite{Dyson-1944} conjectured the existence of such a statistic named crank,
and its definition was discovered by Andrews and Garvan \cite{Andrews-Garvan-1988}.
The crank of a partition was defined as the largest part if the partition contains no ones,
and otherwise as the number of parts larger than the number of ones minus the number
of ones.
For more details, please refer to Andrews and Garvan \cite{Andrews-Garvan-1988,Garvan-1988}.

In view of partition congruences,
Ono further conjectured that the crank functions
enjoy all of the partition congruences.
The work of Mahlburg \cite{Mahlburg-2005,Mahlburg-2006} provided an
elegant answer to these questions.
Let $M(m,n)$ denote the number of partitions of $n$ with crank is $m$.
Then the generating function of $M(m,n)$ can be written as
\begin{align}\label{cg}
\sum_{m=-\infty}^\infty\sum_{n=0}^\infty  M(m,n)z^mq^n
=\frac{(q;q)_\infty}{(zq,z^{-1}q;q)_\infty}.
\end{align}
Mahlburg \cite{Mahlburg-2005,Mahlburg-2006} developed the theory of congruences for crank generating functions, which are shown to possess the same sort of arithmetic properties as partition functions.
For instance, \cite{Mahlburg-2005,Mahlburg-2006} stated that if $\tau$ is a positive integer,
then there are infinitely many non-nested arithmetic
progressions $An+B$ such that
\begin{align*}
M(m,\ell^j,An+B)\equiv0\pmod{\ell^\tau},
\end{align*}
simultaneously for every $0\leq m\leq \ell^j-1$.

Motivated by the work of Mahlburg,
we define {\it crank-type generating function}
as follows:
\begin{align}\label{GF}
\mathfrak{F}_{r,d,t}(z,q)=\sum_{m=-\infty}^\infty \sum_{n=0}^\infty M_{r,d,t}(m,n)z^mq^n
=\frac{(q^{r};q^{r})_\infty^{d}}{(q;q)_\infty^t(zq,q/z;q)_\infty}.
\end{align}
Note that taking $r=t=1$ and $d=2$,
\eqref{GF} reduces to the ordinary crank generating function \eqref{cg}.
When $r=d=t=1$, \eqref{GF} becomes the generating function of the birank of an ordered pair of partitions
given by Hammond and Lewis \cite{Hammond-Lewis-2004}.
Taking $r=t=1$ and $d=k-1$ in \eqref{GF},
we obtain the generating function
introduced by Bringmann and Dousse \cite{Bringmann-Dousse-2016},
whose combinatorial interpretation is given by Fu and Tang \cite{Fu-Tang-2018-1} as the $k$-crank of $k$-colored partition.

Our main result shows that the coefficients of crank-type generating function enjoy the same arithmetic property as partition functions as well as crank.

\begin{thm}\label{main-thm-M}
Let $j,v$ be two positive integers. We fix a prime number $\ell>dr^2$ such that $\left(\frac{g^2d-tr}{\ell}\right)$ are the same for all $g|r$. Then there are infinitely many non-nested arithmetic
progressions $An+B$ such that
\begin{align*}
M_{r,d,t}(m,\ell^j,An+B)\equiv0\pmod{\ell^v},
\end{align*}
simultaneously for every $0\leq m\leq \ell^j-1$.
\end{thm}

Theorem  \ref{main-thm-M} provides the congruence of birank and $k$-crank.
Recall that the birank $b(\pi)$ of an ordered pair of partitions $\pi=(\lambda(1),\lambda(2))$ \cite{Hammond-Lewis-2004}  is the number of parts in the first partition minus the number of parts in the second partition, that is $b(\pi)=\#(\lambda(1))-\#(\lambda(2))$. The number of ordered pairs of partitions of weight $n$ having birank $m$ will be written as $R(m,n)$, hence the generating function for $R(m,n)$ is
\begin{equation*}
\sum_{n=0}^\infty \sum_{m=-\infty}^\infty R(m,n) z^mq^n=\frac{1}{(zq,q/z;q)_\infty}.
\end{equation*}
\begin{cor}
Let $R(m,N,n)$ denote the number of ordered pairs of partitions of weight $n$ with birank congruent to $m$ modulo $N$.
Suppose that $\ell\geq5$ is prime and that $\tau$ and $j$ are positive integers.
Then there are infinitely many non-nested arithmetic
progressions $An+b$ such that
\begin{align*}
R(m,\ell^j,An+B)\equiv0\pmod{\ell^\tau},
\end{align*}
simultaneously for every $0\leq m\leq \ell^j-1$.
\end{cor}

The $k$-colored partition is a $k$-tuple of partitions $\lambda=(\lambda^{(1)},\lambda^{(2)},\ldots,\lambda^{(k)})$. For $k\geq 2$, Fu and Tang \cite{Fu-Tang-2018-1} defined the $k$-crank of $k$-colored partition as follows:
\begin{equation}
k\text{-crank}(\lambda)=\ell(\lambda^{(1)})-\ell(\lambda^{(2)}),
\end{equation}
where $\ell\left(\pi^{(i)}\right)$ denotes the number of parts in $\pi^{(i)}$.
Let $M_k(m,n)$ denote the number of $k$-colored partitions of $n$ with $k$-crank $m$.
The generating function of $M_k(m,n)$ can be derived by Bringmann and Dousse \cite{Bringmann-Dousse-2016}:
\begin{equation*}
\sum_{n=0}^\infty \sum_{m=-\infty}^\infty M_k(m,n) z^mq^n=\frac{(q;q)_\infty^{2-k}}{(zq;q)_\infty (z^{-1}q;q)_\infty}.
\end{equation*}

\begin{cor}
Let $M_k(m,N,n)$ denote the number of $k$-colored partitions of $n$ with $k$-crank congruent to $m$ modulo $N$. Suppose that $\ell\geq5$ is prime and that $\tau$ and $j$ are positive integers.
Then there are infinitely many non-nested arithmetic
progressions $An+b$ such that
\begin{align*}
M_k(m,\ell^j,An+B)\equiv0\pmod{\ell^\tau},
\end{align*}
simultaneously for every $0\leq m\leq \ell^j-1$.
\end{cor}


\section{Preliminaries}\label{Sec-2}
In this section, we present some basic proposition of modular forms of half-integer weight which will be used in the following sections, see \cite{Ono} for details.
Let $f=\sum_{n\geq 0}a_nq^n\in M_k(\Gamma_0(N),\chi)$ be a modular form of weight $k\in\frac{1}{2}\mbz\backslash \mbz$ with Nebentype character $\chi$.

Let $q=e^{2\pi iz}$. If
\[f=\sum_{n=0}^\infty a(n)q^n\in M_k(\Gamma_0(N),\chi)\]
and $\psi$ is a Dirichlet character with modulus $M$, we define the twist of $f$ by $\psi$ to be
\[f(\tau)\otimes\psi=\sum_{n=0}^\infty \psi(n)a(n)q^n.\]
It is well-known that $f(\tau)\otimes \psi\in M_k(NM^2,\chi\psi^2)$.
Fix a constatn $\epsilon$, we define
\[\widetilde{f_{\epsilon,\psi}}(\tau)=f(\tau)-\epsilon f(\tau)\otimes\psi.\]
In the following, when $\epsilon$ and the character $\psi$ is specified, we would write it as $\tilde{f}$.

With the notation above, let $p\nmid N$ be a prime number, then the action of the half-integra Hecke operator $T_{p^2}$ is defined by
\begin{equation*}
    f|T_{p^2}=\sum_{n\geq 0}\left(a(p^2n)+\chi(p)\left(\frac{(-1)^{k-\frac{1}{2}} n}{p}\right)p^{k-\frac{3}{2}}a(n)+\chi(p^2)\left(\frac{(-1)^{k-\frac{1}{2}}}{p^2}\right)p^{2k-2}a\left(\frac{n}{p^2}\right)\right)q^n.
\end{equation*}
We need the following result from Serre \cite{Serre-1976}.
\begin{thm}\label{thm:serre}
Let $f_i(\tau)\in S_{k_i}(\Gamma_1(N_i))$ be half-integer weight cusp forms with algebraic integer coefficients where $k_i\in \frac{1}{2}\mbz\backslash\mbz$ for $i=1,2,\cdots,m$. Then for any $M\geq 1$, there exists a positive proportion of primes $p\equiv -1\pmod{N_1\cdots N_m M}$ such that
\begin{equation*}
    f_i|T_{p^2}\equiv0\pmod{M},
\end{equation*}
for every $i=1,2,\cdots,m$.
\end{thm}

Let $\vec{a}=(a_1,a_2)\in \mbq^2$. We set $\tau'=a_1\tau+a_2$ and $q_{\tau'}=\be(\tau')$. The Klein form is given by
\begin{align}\label{Klein}
    \mathfrak{k}_{\vec{a}}(\tau)=q_{\tau'}^{\frac{a_1-1}{2}}\frac{(q_{\tau'};q)_{\infty}(q/q_{\tau'};q)_{\infty}}{(q;q)_{\infty}^2}.
\end{align}

We recall some basic properties of Klein forms, one can find more details in \cite{Lang1987}.

\begin{prop}\label{prop:klein}
The Klein form $\mathfrak{k}_{\vec{a}}(\tau)$ satisfies the transformation formula, i.e. for any $\gamma\in \SL_2(\mbz)$, we have
\begin{equation*}
    \mathfrak{k}_{\vec{a}}|_\gamma(\tau)=\mathfrak{k}_{\vec{a}\cdot\gamma}(\tau).
\end{equation*}
If $(n_1,n_2)\in\mbz^2$, then we have
\begin{equation*}
    \mathfrak{k}_{\vec{a}+(n_1,n_2)}(\tau)=(-1)^{n_1+n_2+n_1n_2}\be\left(\frac{a_1n_2-a_2n_1}{2}\right)\mathfrak{k}_{\vec{a}}(\tau).
\end{equation*}
\end{prop}

\begin{lem}\label{lem:etaell}
Let $\ell$ be a prime number, then we have
\begin{equation}
    \frac{\eta^{\ell}(\ell\tau)}{\eta(\tau)}\in M_{\frac{\ell-1}{2}}\left(\Gamma_0(\ell),\left(\frac{\ast}{\ell}\right)\right).
\end{equation}
\end{lem}

\begin{lem}\label{lem:Et}
For integer $j\geq 2$, let
\begin{equation*}
    E_j(\tau)=\frac{\eta^{\ell^j}(\tau)}{\eta(\ell^j\tau)}.
\end{equation*}
Then we have $E_j(\tau)\in M_{\frac{\ell^j-1}{2}}(\Gamma_0(\ell^j),\chi_j)$ where $\chi_j=\left(\frac{(-1)^{\frac{\ell^j-1}{2}}\ell^j}{\ast}\right)$. Moreover, $E_j(\tau)$ vanishes at every cusp $\frac{a}{c}$ with $\ell^j$ not dividing $c$.
\end{lem}

\section{Proof of Theorem \ref{main-thm-M}}\label{Sec-3}

The following lemmas play central roles in the proof of Theorem \ref{main-thm-M}.
\begin{lem}\label{lem:gausssum}
Let $e_{\ell}=\sum_{n=1}^{\ell-1}\left(\frac{n}{\ell}\right)e^{2\pi in/\ell}$ be the Gauss sum. Then the twisted modular form has the following expression.
\begin{align}\label{lem:gausssum-eq}
f(z)\otimes \left(\frac{\ast}{\ell}\right)=\frac{e_{\ell}}{\ell}\sum_{n=1}^{\ell-1}\left(\frac{n}{\ell}\right)f(\tau)\left|\begin{pmatrix}
1&-\frac{n}{\ell}\\0&1
\end{pmatrix}\right..
\end{align}
\end{lem}

\begin{proof}
The right hand side of \eqref{lem:gausssum-eq} can be written as
\begin{align*}
    &\frac{e_{\ell}}{\ell}\sum_{n=1}^{\ell-1}\sum_{m\geq 1}a(m)\left(\frac{n}{\ell}\right)\be\left(2\pi im\left(\tau-\frac{n}{\ell}\right)\right)\\
    &\quad=\frac{e_{\ell}}{\ell}\sum_{m\geq 1}e_{\ell}a(m)\left(\frac{-m}{\ell}\right)q^m\\
    &\quad=\sum_{m\geq 1}a(m)\left(\frac{m}{\ell}\right)q^m,
\end{align*}
which completes the proof.
\end{proof}
Consider a positive integer $N$ and set $\zeta=e^{2\pi i/N}$ in \eqref{GF}.
For any residue class $m \pmod{N}$, elementary calculations give the generating function for the crank
\begin{align}\label{GF-M}
\sum_{n=0}^\infty M_{r,d,t}(m,N,n)q^n
&=\frac{1}{N}\sum_{s=0}^{N-1}\mathfrak{F}_{r,d,t}(\zeta^s,q)\zeta^{-ms}
\nonumber\\[3pt]
&=\frac{1}{N}\sum_{s=0}^{N-1}\zeta^{-ms}
\left(\prod_{n=1}^\infty\frac{(1-q^{rn})^{d}}{(1-q^n)^{t}(1-\zeta^sq^n)(1-\zeta^{-s}q^n)}\right).
\end{align}
Substituting the definition of the Klein form \eqref{Klein} into \eqref{GF-M}, we consider the following series
\begin{equation*}
\begin{split}
    g_m(\tau)&=\frac{1}{2\pi i}\sum_{s=1}^{N-1}\frac{\omega_s\zeta^{-ms}}{\mathfrak{k}_{(0,\frac{s}{N})}(\tau)}\frac{\eta^{t\ell}(\ell\tau)\eta^{d\ell^v}(r\tau)}{\eta^{t}(\tau)}+\frac{\eta^{t\ell}(\ell\tau)\eta^{d\ell^v}(r\tau)}{\eta^{t}(\tau)}\\
    &:=G_m(\tau)+P(\tau)
    \end{split}
\end{equation*}
where $\omega_s=\zeta^{s/2}(1-\zeta^{-s})$.

\begin{lem}\label{lem:gm}
Let $r,d,t,\ell,v$ defined as before, then we have
\begin{equation}
    \frac{\widetilde{G_m}(24\tau)}{\eta^{t\ell}(24\ell \tau)\eta^{d \ell^v}(24r\tau)}E_{j+1}(24\tau)^{\ell^v}\in S_{\lambda+1/2}(\Gamma_0(576r\ell^{j+1}),\chi)
\end{equation}
for some integer $\lambda$.
\end{lem}

\begin{proof}
By the Lemma \ref{lem:Et} and the Theorem 1.65 of \cite{Ono}, we can show that the order of $E_{j+1}(\tau)$ at the cusp $\frac{a}{c}$ with $\ell^{j+1}$ not dividing $c$ is at least $\frac{\ell^2-1}{24}$. On the other hand, we have
\begin{equation*}
    \ord_{\frac{a}{c}}\eta^{t\ell}(\ell\tau)\eta^{d\ell^v}(r\tau)=\frac{t\ell^3+dr^2\ell^v}{24}< \frac{\ell^2-1}{24}\ell^v.
\end{equation*}

Hence it is enough to show that $\frac{\widetilde{G_m}(\tau)}{\eta^{k\ell}(\ell\tau)\eta^{d\ell^v}(r\tau)}$ vanishes at each cusp $\frac{a}{c}$ with $\ell^{j+1}|c$. Now we choose suitable integers $b,d$ such that $A=\begin{pmatrix}a&b\\c&d\end{pmatrix}\in \Gamma_0(\ell^{j+1})$.

Now we compute the order of $G_m$ at the cusp $\frac{a}{c}$. To do that, let $g=\gcd(c,r)$, $a_1=ra/g$, $c_1=c/g$. It is easy to see $\gcd(a_1,c_1)=1$, so we can choose integers $b_1,d_1$ such that $B=\begin{pmatrix}a_1&b_1\\c_1&d_1\end{pmatrix}\in \SL_2(\mbz)$. Then
\begin{align*}
&\Delta(r\tau)|_{12}A=(c\tau+d)^{12}\Delta\left(\frac{ar\tau+br}{c\tau+d}\right)\\
=&(c\tau+d)^{12}\Delta\left(B\left(\frac{g\tau+h}{r/g}\right)\right)\\
=&\Delta\left(\frac{g\tau+h}{r/g}\right),
\end{align*}
where $h=rbd_1-b_1d$. On the other hand, by the Lemma \ref{lem:etaell}, we see that $\frac{\eta^\ell(\ell\tau)}{\eta(\tau)}$ belongs to $M_{\frac{\ell-1}{2}}\left(\Gamma_0(\ell),\left(\frac{\ast}{\ell}\right)\right)$. By Proposition \ref{prop:klein}, we have
\begin{equation*}
    \mathfrak{k}_{(0,\frac{s}{N})}|A(\tau)=\be\left(\frac{N(cs+ds-\overline{ds})+cds^2}{2N^2}\right)
    \mathfrak{k}_{(0,\frac{\overline{ds}}{N})}(\tau),
\end{equation*}
where $\overline{ds}$ is the unique integer between $0$ and $N-1$ such that $N|ds-\overline{ds}$.
Hence combining the discussion above, we get
\begin{equation*}
    G_m(\tau)|A=\frac{1}{2\pi i}\sum_{s=1}^{N-1}\frac{w_s\zeta^{-ms}}{e_N(c,d,s)
    \mathfrak{k}_{(0,\frac{\overline{ds}}{N})}(\tau)}\left(\frac{d}{\ell}\right)^t
    \frac{\eta^{t\ell}(\ell\tau)}{\eta^t(\tau)}\Delta^{(d\ell^v+d)/24}\left(\frac{g\tau+h}{r/g}\right),
\end{equation*}
where $e_N(c,d,s)=\be\left(\frac{N(cs+ds-\overline{ds})+cds^2}{2N^2}\right)$.
Next we consider $G_m\otimes \left(\frac{\ast}{\ell}\right)$. For integer $0\leq u\leq \ell-1$, we choose $u'$ such that $u'\equiv d^2u\pmod{\ell}$. Then we have
\begin{equation*}
    \begin{pmatrix}1&-u/\ell\\ 0&1\end{pmatrix}\begin{pmatrix}a&b\\ c&d\end{pmatrix}=\begin{pmatrix}a-cu/\ell&b-cuu'/\ell^2+(au'-du)/\ell\\ c&d+cu'/\ell\end{pmatrix}\begin{pmatrix}1&-u'/\ell\\ 0&1\end{pmatrix}
\end{equation*}
With the help of Lemma $\ref{lem:gausssum}$ we can show that
\begin{equation*}
    G_m\otimes \left(\frac{\ast}{\ell}\right)|A=\frac{e_\ell}{2\pi i\ell}\sum_{s=1}^{N-1}\sum_{u=1}^{\ell}\frac{w_s\zeta^{-ms}}{e_N(c,d',s)'\mathfrak{k}_{(0,\frac{\overline{d's}}{N})}(\tau)}\left(\frac{d'^tu}{\ell}\right)\frac{\eta^{t\ell}(\ell\tau)}{\eta^t(\tau)}\Delta^{(d\ell^v+d)/24}\left(\frac{g\tau+h_s}{r/g}\right)
\end{equation*}
where $d'=d+cu'/\ell$ and $h_s$ are integers. We compare the first nonzero coefficient of $G_m(\tau)|A$ and $G_m\otimes \left(\frac{\ast}{\ell}\right)|A$. In fact, if we assume that
\begin{equation*}
    G_m(\tau)|A=a_nq^n+a_{n+1}q^{n+1}+\cdots
\end{equation*}
where $n=t\delta_{\ell}+\frac{g^2(d\ell^v+d)}{24r}$. Then since $r$ divides $d\ell^v+d$, the expansion of $G_m\otimes \left(\frac{\ast}{\ell}\right)|A$ is
\begin{equation*}
    a_n\left(\frac{n}{\ell}\right)q^n+a_{n+1}'q^{n+1}+\cdots
\end{equation*}
By the assumption on $\ell$, this implies that
\begin{equation*}
    \ord_{\frac{a}{c}}\widetilde{G_m}(\tau)\geq t\delta_{\ell}+\frac{g^2(d\ell^v+d)}{24r}+1.
 \end{equation*}
By comparing the order, we see that $\frac{\widetilde{G_m}(\tau)}{\eta^{k\ell}(\ell\tau)\eta^{d\ell^v}(r\tau)}$ also vanishes at the cusp $\frac{a}{c}$ when $\ell^{j+1}|c$.
\end{proof}

The proof of following lemma is similar to Lemma \ref{lem:gm}, so we omit it.
\begin{lem}\label{lem:p}
Let $v$ be an integer large enough, then we have
\begin{equation}
    \frac{\tilde{P}(24\tau)}{\eta^{t\ell}(24\ell \tau)\eta^{d\ell^v}(24r\tau)}E_{j+1}(24\tau)^{\ell^v}\in S_{\lambda'+1/2}(\Gamma_0(576\ell^{\max\{3,j+1\}}),\chi).
\end{equation}
for some integer $\lambda'$.
\end{lem}

\begin{lem}\label{lem:subseq}
Let $f(q)$ be a formal power series such that $f(q)\equiv \sum_{n\geq 0}a_nq^{n\ell}\pmod{\ell^v}$. If the series $\sum b_nq^n$ and $\sum b_n'q^n$ coincide on the subsequence $\{\ell m+d\, |\,m\in \mbz\}$, then the series $\sum b_nq^nf(q)$ and $\sum b_n'q^nf(q)$ also coincide on the subsequence $\{\ell m+d\,|\,m\in \mbz\}$ modulo $\ell^v$.
\end{lem}

\begin{proof}
By the assumption, $\ell^v$ divides the coefficient of $q^n$ of $f(q)$ when $\ell\nmid n$. So in the expression of $\sum b_nq^nf(q)$, the coefficient of $q^{\ell m+d}$ is congruence to $\sum b_{\ell n+d}a_{m-n}\pmod{\ell^v}$. This implies that $\sum b_nq^nf(q)$ and $\sum b_n'q^nf(q)$ coincide on the subsequence $\{\ell m+d\,|\,m\in \mbz\}$ modulo $\ell^v$.
\end{proof}

The following theorem gives the precise congruence relations for $M_{r,d,t}$,which implies Theorem \ref{main-thm-M}.

\begin{thm}
Let $r,d,t,j,v$ be positive integers and $\ell>dr^2$ be a prime number such that $\left(\frac{g^2d-tr}{\ell}\right)$ are the same for all $g|r$. Then there exists a positive proportion of primes $p\equiv -1\pmod{24\ell}$ such that
\begin{equation}\label{eq:mrdt}
    M_{r,d,t}\left(m,\ell^j,\frac{np^3+k-dr}{24}\right)\equiv 0\pmod{\ell^{v}}
\end{equation}
for all $0\leq m<\ell^j$ and $n\equiv k-dr-24\beta\pmod{24\ell}$ with $p\nmid n$ where $\beta$ is a certain integer between $0$ and $\ell$.
\end{thm}

\begin{proof}
We first assume that $v\geq \max\{5,t\}$ satisfies $\ell^v\equiv -1\pmod{24r}$. We set $\epsilon=\left(\frac{24(dr-t)}{\ell}\right)$, $\alpha=\frac{t(\ell^2-1)+dr(\ell^v+1)}{24}$ and $0\leq \beta<\ell$ such that $\left(\frac{\alpha+\beta}{\ell}\right)=0$ or $-\epsilon$. By the definition of $\widetilde{g_m}(\tau)$, it is easy to that $g_m(\tau)$ coincides with $\widetilde{g_m}(\tau)$ on the subsequent $\{\ell n'+\alpha+\beta\}$. Moreover, we have  \begin{equation*}
    \eta^{d\ell^v}(24r\tau)=q^{dr\ell^v}\prod_{n\geq 1}(1-q^{24rn})^{d\ell^v}\equiv q^{dr\ell^v}\prod_{n\geq 1}(1-q^{24rd\ell^vn})\pmod{\ell^v}
\end{equation*}
Hence by the Lemma \ref{lem:subseq}, we see that the subsequence of $\frac{\widetilde{g_m}(24\tau)}{\eta^{t\ell}(24\ell \tau)\eta^{d \ell^v}(24r\tau)}$ with the indices $\{\ell n'+\alpha+\beta\}$ coincides with $\frac{g_m(24\tau)}{\eta^{t\ell}(24\ell \tau)\eta^{d \ell^v}(24r\tau)}$ on the indices $\{\ell n'+\alpha+\beta\}$ which is exactly
\begin{equation*}
    \sum_{n'\equiv \alpha+\beta\pmod{\ell}}\ell^jM_{r,d,t}(m,\ell^j,n'-\alpha)q^{24n'-k\ell^2-dr\ell^v}\pmod{\ell^v}
\end{equation*}
By changing the variable, we get
\begin{equation*}
    \sum_{n\equiv 24\beta+dr-k\pmod{24\ell}}\ell^jM_{r,d,t}\left(m,\ell^j,\frac{n+k-dr}{24}\right)q^n.
\end{equation*}
Combining Lemma \ref{lem:gm} and Lemma \ref{lem:p}, we see that there exist two modular forms $f_m$ and $f$ such that
\begin{equation*}
    \sum_{n\equiv 24\beta+dr-k\pmod{24\ell}}\ell^jM_{r,d,t}\left(m,\ell^j,\frac{n+k-dr}{24}\right)q^n\equiv f_m(\tau)+f(\tau)\pmod{\ell^v}.
\end{equation*}
Moreover, by applying Theorem \ref{thm:serre}, we can find a positive propotion of primes $p\equiv -1\pmod{24\ell}$ such that
\begin{equation*}
    f_m|T_{p^2}=f|T_{p^2}\equiv 0\pmod{\ell^v}.
\end{equation*}
This implies that
\begin{equation*}
\begin{split}
    &\ell^jM_{r,d,t}\left(m,\ell^j,\frac{p^2n+k-dr}{24}\right)+\ell^j\chi(p)\left(\frac{(-1)^\lambda n}{p}\right)p^{\lambda-1}M_{r,d,t}\left(m,\ell^j,\frac{n+k-dr}{24}\right)\\
    &+\ell^j\chi(p^2)\left(\frac{(-1)^\lambda}{p^2}\right)p^{2\lambda-1}M_{r,d,t}\left(m,\ell^j,\frac{n/p^2+k-dr}{24}\right)\equiv 0\pmod{\ell^v}.
    \end{split}
\end{equation*}
Finally, replacing $n$ by $pn'$ with $p\nmid n'$, we get
\begin{equation*}
    \ell^jM_{r,d,t}\left(m,\ell^j,\frac{p^3n'+k-dr}{24}\right)\equiv 0\pmod{\ell^v}
\end{equation*}
where $n'\equiv k-dr-24\beta\pmod{24\ell}$. This gives
\begin{equation*}
    M_{r,d,t}\left(m,\ell^j,\frac{p^3n'+k-dr}{24}\right)\equiv 0\pmod{\ell^{v-j}}
\end{equation*}
But we note that there are infinite many $v$ satisfies $\ell^v\equiv -1\pmod{24r}$, so the congruence relation $\eqref{eq:mrdt}$ holds for all $v\in\mbn$.
\end{proof}

We end this article with a short list of questions arising from this project. A natural question is that can one establish the combinatorial interpretations of $M_{r,d,t}(m,n)$?
Moreover, can we find any other kind of congruence of $M_{r,d,t}(m,n)$?

\vspace{0.5cm}
 \baselineskip 15pt
{\noindent\bf\large{\ Acknowledgements}} \vspace{7pt} \par
The first author would like to acknowledge that the research was supported by Fundamental Research Funds for the Central Universities (Grant No. 531118010622). The second author would like to acknowledge that the research was supported by the National Natural Science Foundation of China (Grant Nos. 12001182 and 12171487), the Fundamental Research Funds for the Central Universities (Grant No. 531118010411) and Hunan Provincial Natural Science Foundation of China (Grant No. 2021JJ40037).


\begin{thebibliography}{99}

\bibitem{Ahlgren-Ono-2001}
S.Ahlgren and K. Ono, \emph{Congruence properties for the partition function}, Proc. Natl. Acad. Sci. USA 98 (23) (2001) 12882--12884.

\bibitem{Andrews-Ono-2005}
G.E. Andrews and K. Ono, \emph{Ramanujan's congruences and Dyson's crank},
Proc. Natl. Acad. Sci. USA 102 (43) (2005) 15277.

\bibitem{Andrews-Garvan-1988}
G.E. Andrews and F.G. Garvan, \emph{Dyson's crank of a partition}, Bull. Amer. Math. Soc. 18 (2) (1988) 167--171.

\bibitem{Bringmann-Dousse-2016}
K. Bringmann and J. Dousse, \emph{On Dyson's crank conjecture and the uniform asymptotic behavior of certain inverse theta functions}, Trans. Amer. Math. Soc. 368 (5) (2016) 3141--3155.

\bibitem{Dyson-1944}
F.J. Dyson, \emph{Some guesses in the theory of partitions}, Eureka (Cambrige) 8 (1944) 10--15.

\bibitem{Folsom-2014}
A. Folsom, \emph{Mock modular forms and $d$-distinct partitions}, Adv. Math. 254 (2014) 682--705.

\bibitem{Fu-Tang-2018-1}
S. Fu and D. Tang, \emph{On a generalized crank for $k$-colored partitions}, J. Number Theory 184 (2018) 485--497.

\bibitem{Garthwaite-Jameson-2021}
S.A. Garthwaite and M. Jameson, \emph{Incongruences for modular forms and applications to partition functions}, Adv. Math. 376 (2021) 107448 17 pp.

\bibitem{Garvan-1988}
F.G. Garvan, \emph{New combinatorial interpretations of Ramanujan's partition congruences mod $5$, $7$ and $11$},
Trans. Amer. Math. Soc. 305 (1) (1988) 47--77.

\bibitem{Hammond-Lewis-2004}
P. Hammond and R. Lewis, \emph{Congruences in ordered pairs of partitions}, Int. J. Math. Math. Sci. (45-48) (2004) 2509--2512.

\bibitem{Lang1987}
S. Lang, \emph{Elliptic Functions}, Springer-Verlag, New York, 1987.

\bibitem{Mahlburg-2005}
K. Mahlburg, \emph{Partition congruences and the Andrews-Garvan-Dyson crank}, Proc. Natl. Acad. Sci. USA 102 (43) (2005) 15373--15376.

\bibitem{Mahlburg-2006}
K. Mahlburg, \emph{Congruences for the coefficients of modular forms and applications to number theory}, Thesis (Ph.D.)-The University of Wisconsin- Madison. 2006. 55 pp.

\bibitem{Ono-2000}
K. Ono, \emph{Distribution of the partition function modulo $m$}, Ann. of Math. (2) 151 (1) (2000) 293--307.

\bibitem{Ono}
K. Ono, \emph{The Web of Modularity: Arithmetic of the Coefficients of Modular Forms and $q$-series}. Conference Board of Mathematical Sciences 102 (2003).

\bibitem{Radu-2012}
C.-S. Radu, \emph{A proof of Subbarao's conjecture}, J. Reine Angew. Math. 672 (2012) 161--175.

\bibitem{Serre-1976}
J.-P. Serre,  \emph{Divisibilit\`{e} de certaines fonctions arithm\`{e}tiques}, Enseign. Math. (2) 22 (3-4) (1976) 227--260.



\end{thebibliography}
\end{document}